\documentclass[twocolumn]{autart}    

\newtheorem{remark}{Remark}
\newtheorem{Prop}{Proposition}

\usepackage{amsmath,amsfonts,amssymb,mathtools}
\usepackage{bbm}
\usepackage{graphicx}
\usepackage{exscale}

\newcommand{\E}[1]{\mathop{{\rm \bf E}\left\{#1\right\}}\nolimits}
\newcommand{\var}[1]{\mathop{{\rm \bf var}\!\left\{#1\right\}}\nolimits}
\DeclareMathAlphabet\mathbfcal{OMS}{cmsy}{b}{n}

\usepackage{clrscode}

\begin{document}

\begin{frontmatter}
\runtitle{SVD-based Cubature Kalman Filtering}  

\title{SVD-based factored-form Cubature Kalman Filtering for continuous-time stochastic systems with discrete measurements} 

\thanks[footnoteinfo]{This paper was not presented at any IFAC
meeting. Corresponding author M.~V.~Kulikova.}

\author[IST]{Maria V. Kulikova}\ead{maria.kulikova@ist.utl.pt},
\author[IST]{Gennady Yu. Kulikov}\ead{gkulikov@math.ist.utl.pt}

\address[IST]{CEMAT, Instituto Superior T\'{e}cnico, Universidade de Lisboa,
          Av. Rovisco Pais 1,  1049-001 Lisboa, Portugal}


\begin{abstract}
In this paper, a singular value decomposition (SVD) approach is developed for implementing the cubature Kalman filter. The discussed estimator is one of the most popular and widely used method for solving nonlinear Bayesian filtering problem in practice. To improve its numerical stability (with respect to roundoff errors) and practical reliability of computations, the SVD-based methodology recently proposed for the classical Kalman filter is generalized on the nonlinear filtering problem. More precisely, we suggest the SVD-based solution for the continuous-discrete cubature Kalman filter and design two estimators: (i) the filter based on the traditionally used Euler-Maruyama discretization scheme; (ii) the estimator based on advanced It\^{o}-Taylor expansion for discretizing the underlying stochastic differential equations. Both estimators are formulated in terms of SVD factors of the filter error covariance matrix and belong to the class of stable factored-form (square-root) algorithms. The new methods are tested on a radar tracking problem.
\end{abstract}

\end{frontmatter}

\section{Introduction}\label{introduction}

This paper deals with the nonlinear Bayesian filtering problem, i.e. with the methods developed for estimating the unknown dynamic state of nonlinear continuous-discrete stochastic systems. More precisely, we explore the {\it local} approach discussed in~\cite{2009:Haykin} where
the nonlinear filters are designed by fixing the posterior density to take {\it a priori} form (it is typically assumed to be Gaussian). This makes their derivation simple and the resulted estimators are fast to execute, although they yield the sub-optimal solution to the nonlinear Bayesian filtering problem.
However, the attractive simplicity of local filters, their good estimation quality and practical feasibility make these estimators are the methods of choice for solving real-time practical problems. Among such estimators, the most popular ones are the extended Kalman filter (EKF)~\cite{1964:Ho,1970:Jazwinski:book}, the divided difference filter (DD KF)~\cite{2000:Norgaard}, the unscented Kalman filter (UKF)~\cite{2000:Julier,2004:Julier,2015:Menegaz}, the quadrature Kalman filter (QKF)~\cite{2007:Arasaratnam,2000:Ito} and the cubature Kalman filter (CKF)~\cite{2009:Haykin,2010:Haykin,2018:Haykin,2014:Wang:CKF} as well as their adaptive versions~\cite{2016:Wang:TAC}. All mentioned methods have the Kalman filtering (KF)-like structure that is constituted by a recursive computation of the first two moments (i.e. the mean and covariance), only. In~\cite{2009:Haykin} this estimation strategy is called the `nonlinear filtering through linear estimation theory'. Nowadays, the sampled-data UKF and CKF are the most widely used nonlinear Bayesian estimators for solving practical applications, e.g. see~\cite{2003:LaViola,2008:Teixeira} and many others. It is justified by the third order approximation for the posterior mean under the sampled-data UKF and CKF strategies compared with the first order approximation under the EKF  methodology~\cite{2013:Sarkka:book,2001:Wan:book}. Among these two sampled-data filtering approaches, the CKF technique is a popular alternative to the UKF methodology because of square-rooting problem related to possibly negative weights in the UKF. This prevents the design of its numerically stable and reliable implementation methods; see the discussion in~\cite[Section~VII]{2010:Haykin} and~\cite{2001:Merwe}.
Additionally, the CKF is shown to be a valuable method when estimating the unknown dynamic state of nonlinear continuous-discrete stochastic systems in the presence of both Gaussian noise~\cite{2010:Haykin,2016:Kulikov:IEEE,2017:Kulikov:ANM} and non-Gaussian uncertainties including outliers~\cite{2018:Kulikov:SP}.

This paper is focused on implementation issues of the continuous-discrete CKF technique~\cite{2010:Haykin}.
Recall, the conventional CKF algorithm requires the Cholesky decomposition of the filter error covariance
for generating the cubature vectors in each iterate.
The decomposition exists and is unique when the symmetric matrix to be decomposed is positive definite~\cite{1983:Golub:book}.
This is not always the case in a finite precision arithmetics and, hence, it may result in unexpected interruption of the computations. To enhance the reliability and numerical robustness to roundoff errors, the Cholesky-based (square-root) CKF implementations have been designed for both the covariance~\cite{2009:Haykin,2010:Haykin} and information cubature filtering~\cite{2017:Arasaratnam,2013:Chandra}. The square-root (SR) approach implies one Cholesky decomposition at the initial filtering step and, next, the CKF equations are re-derived in such a way that the only Cholesky factors are propagated. This mechanization ensures the positive (semi-) definiteness and symmetric form of the resulted error covariance matrix and yields the improved numerical stability with respect to roundoff error; see more detail in~\cite[Chapter~7]{2015:Grewal:book}. The SR filters belong to the class of {\it factored-form} implementation methods~\cite{2010:Grewal}. The comprehensive overview of existing CKF implementations can be found in~\cite[Chapter~6]{2019:Chandra:book}.

Recent results in the filtering realm suggest that the singular value decomposition (SVD) utilization for generating the sigma vectors in the UKF estimators~\cite{2013:Straka,2016:Vsimandl} as well as the cubature nodes in the CKF methodology~\cite{2015:Zhang} sometimes yields a better estimation quality compared to the traditionally used Cholesky decomposition; see the cited papers and the discussion in~\cite{2015:Menegaz}. However, the mentioned filters still imply the recursive computation of the entire error covariance matrix, i.e. they are the {\it conventional}-form implementations with  a poor numerical stability in a finite precision arithmetics; e.g. see the discussion of estimation problems with the ill-conditioned measurement scenario in~\cite[Examples 7.1 and 7.2]{1969:Dyer} and many other studies. Undoubtedly, it is reasonable to design inherently stable Bayesian filters in terms of updating the SVD factors of the filter error covariance matrix. This is the challenge to be addressed in this brief paper.
Compared with the existing Cholesky-based CKF solution with the implied triangular SR matrix factors, the new SVD factorization-based approach yields a more general scenario. The SR factors obtained under the SVD strategy are full matrices that might be rectangular ones, in general. Indeed, the SVD factorization provides useful information about the involved matrix structure and this knowledge about the eigenvalues can be  intelligently used for a related low-rank approximation providing the reduced-rank filters with a decreased computational cost. Another benefit of the SVD factorization-based filtering strategy is its practical feasibility for the case of `almost exact' or `perfect' measurements, i.e. when the measurement noise covariance matrix is positive semi-definite~\cite[p.~1305]{2010:Simon}. It comes from the fact that SVD exists for any matrix that is not a case for Cholesky decomposition. Finally, {\it information}-form CKF implementations are easy to get within SVD strategy because the inverse of the involved orthogonal SVD matrix factor is its transpose one, meanwhile the diagonal SVD factor inversion requires a few scalar divisions, only. Thus, it is straightforward to design the mixed-type implementations with automatic switching between the covariance-form and information-form regimes. This feature is of special interest for practical applications where the number of states to be estimated is much less than the dimension of measurement vector $m>>n$. In this case the measurement update step in its traditional {\it covariance}-form formulation is extremely expensive because of the required $m$-by-$m$ matrix inversion, but the related {\it information}-form is simple~\cite[Section~7.7.3]{2015:Grewal:book}. Hence, a feasibility of the switching regime may be used for a computational cost reduction. A clear example of applications where information-form algorithms greatly benefit is the distributed multi-sensor architecture applications with
a centralized KF used for processing the information coming from a large number of sensors~\cite{2010:Sayed}.

Motivated by the advanced functionality of any SVD-based filter, the goal is to derive the first SVD-based CKF algorithms. By now, the existing SVD-based estimators are restricted by linear dynamic systems, only~\cite{2017:Kulikova:IET,2019:Galka,1992:Wang}.
Here, we propose two nonlinear estimators: (i) the filter based on the traditionally used Euler-Maruyama discretization scheme of order~0.5; (ii) the estimator based on advanced It\^{o}-Taylor expansion of order~1.5 for discretizing the underlying stochastic differential equations (SDEs). Both filters are formulated in terms of SVD factors of the filter covariance and belong to the class of factored-form algorithms with intrinsic numerical robustness. In our derivation we follow the {\it discrete-discrete} approach for continuous-discrete Bayesian filtering that means that the SDE discretization schemes~\cite{1999:Kloeden:book} are used first, and then the corresponding moments' equations are derived. Recall, there exists an  alternative strategy for designing the continuous-discrete nonlinear Bayesian filters. The so-called {\it continuous-discrete} approach implies the derivation of the corresponding moment differential equations first, and then the numerical integration methods derived for solving ordinary differential equations (ODE) are applied for solving the resulted moment equations; e.g. see the continuous-discrete UKF in~\cite{2007:Sarkka}, the continuous-discrete CKF in~\cite{2012:Sarkka} and the continuous-discrete EKF variants in~\cite{2016:Kulikov:IEEE,2017:Kulikov:ANM}. These two methodologies are  extensively discussed  in~\cite{2012:Frogerais,2014:Kulikov:IEEE}.
Finally, the newly-developed SVD-based CKF methods are tested on a radar tracking problem.

\section{Conventional Cubature Kalman filter} \label{problem:statement}

Consider continuous-discrete stochastic model
\begin{align}
dx(t) & = f\bigl(t,x(t)\bigr)dt + \tilde Gd\beta(t), \quad t>0,  \label{eq1.1} \\
z_k   & =  h(k,x(t_{k}))+v_k, \quad k =1,2,\ldots \label{eq1.2}
\end{align}
where  $x(t) \in \mathbb R^{n}$ is the unknown state vector to be estimated and the vector-function $f:\mathbb R\times\mathbb
R^{n}\to\mathbb R^{n} $ is the time-variant drift function. The process uncertainty is modelled by the additive noise term where
$\tilde G \in \mathbb R^{n\times n}$ is the time-invariant diffusion matrix and $\beta(t)$ is
the $n$-dimensional Brownian motion whose increment $d\beta(t)$ is independent of $x(t)$ and has the covariance $Q\,dt>0$. Finally, measurement equation~\eqref{eq1.2} implies a nonlinear relationship $h:\mathbb R\times\mathbb R^{n}\to\mathbb R^{m}$  between the unknown dynamic state $x(t_{k})$ and available data $z_k := z(t_{k})$ measured at some discrete-time points $t_k$. This measured information $z_k \in \mathbb R^{m}$ comes with the sampling rate (sampling period) $\Delta_k=t_{k}-t_{k-1}$. The measurement noise term $v_k $ is assumed to be white Gaussian noise with zero mean and known covariance $R_k>0$. The  initial state $x(t_0)$ is assumed to be normally distributed with mean $\bar x_0$ and covariance $\Pi_0>0$. Finally, the noise processes and $x(t_0)$ are all assumed to be statistically independent.

The CKF technique for stochastic systems~\eqref{eq1.1}, \eqref{eq1.2} suggested in~\cite{2010:Haykin} is developed within the {\it discrete-discrete} framework and implies the following two steps. First, the underlying SDE system in~\eqref{eq1.1} is solved. For that, each sampling interval $\Delta_k=t_{k}-t_{k-1}$ is divided into $M-1$ subintervals and the stochastic system is discretized on each subinterval for obtaining the approximated SDEs solution. Following~\cite{2010:Haykin}, the key feature of this step is the use of discretization method of order 1.5 based on It\^{o}-Taylor expansion (IT-1.5). Recall, a lower order method, which is the Euler-Maruyama discretization scheme of order 0.5 (EM-0.5), is traditionally used in engineering literature~\cite{1970:Jazwinski:book}. It is easy to integrate with the CKF methodology as well~\cite{2018:Kulikov:Romania}. Although the method yields a higher discretization error compared to the IT-1.5 scheme for the same number of subintervals utilized, its simplicity might be an attractive feature for practitioners. Thus, the EM-0.5 CKF estimator deserves some merit, and we propose our new SVD-based solution for both the IT-1.5 and EM-0.5 CKF estimators. The second step in designing the filtering method is the mean and covariance approximation. The key feature of this step is to approximate the expectations in the resulted equations with the third order cubature integration method.

We briefly explain the IT-1.5 discretization scheme utilized in~\cite{2010:Haykin}. The method is used on the equidistant mesh, i.e. the length of each sampling interval is assumed to be constant (i.e. $\Delta = \Delta_k = \Delta_{k-1} \ldots$). More precisely, each interval is divided on the pre-defined $M-1$ equally spaced subdivision nodes. Hence, the step size of the IT-1.5 scheme is $\delta=\Delta/M$ and we have $t_{k-1}^{(m)} = t_{k-1} + m \delta$, $m=0, \ldots, M-1$. The noise term in~\eqref{eq1.1} is assumed to be standard Brownian motion. If $Q \ne I$, one may define $G:=\tilde GS_{Q}$ where $S_Q$ is a square-root factor of the process covariance $Q = S_QS_Q^{\top}$.
Following~\cite{2010:Haykin}, the underlying SDEs are discretized on each $[t_{k-1}, t_k]$) as follows:
\[
x_{k-1}^{(m+1)} =f_d\bigl(t_{k-1}^{(m)},x_{k-1}^{(m)}\bigr)+G w_1 +{\mathbb L}f\bigl(t_{k-1}^{(m)},x_{k-1}^{(m)}\bigr) w_2
\]
where $w_1$ and $w_2$ are the correlated $n$-dimensional zero-mean random Gaussian variables with the properties discussed in~\cite{2010:Haykin,1999:Kloeden:book}. The discretized drift function is
\begin{align}
f_d\bigl(t_{k-1}^{(m)},x_{k-1}^{(m)}\bigr) & = x_{k-1}^{(m)}+\delta f\bigl(t_{k-1}^{(m)},x_{k-1}^{(m)}\bigr) \nonumber \\
 & + \frac{1}{2}{\delta^2}{\mathbb L}_0 f\bigl(t_{k-1}^{(m)},x_{k-1}^{(m)}\bigr) \label{fun_fd}
\end{align}
 where ${\mathbb L} f$ stands for a square matrix with $(i,j)$ entry being ${\mathbb L}_jf_i$ and the utilized differential operators ${\mathbb L}_0$ and ${\mathbb L}_j$ are defined as follows:
\begin{align}
{\mathbb L}_0 & = \frac{\partial}{\partial t} + \sum \limits_{i = 1}^{n} f_i \frac{\partial}{\partial x_i} + \frac{1}{2}\sum \limits_{j,p,r=1}^{n} G_{pj}  G_{rj} \frac{\partial^2}{\partial x_p \partial x_r}, \label{operator_L0} \\
{\mathbb L}_j &= \sum \limits_{i = 1}^{n} G_{ij} \frac{\partial}{\partial x_i}, \quad i,j = 1, 2,\ldots, n. \label{operator_Lj}
\end{align}

It is worth noting here that the IT-1.5 CKF proposed in~\cite{2010:Haykin} is valid for stochastic systems with time-invariant diffusion matrix, only. If the diffusion $G(\cdot)$ related to the standard Brownian motion $\beta(t)$ in the examined state-space model is time-variant, then the differential operators ${\mathbb L}_0$ and ${\mathbb L}_j$ have a more sophisticated expression than formulas~\eqref{operator_L0}, \eqref{operator_Lj}, i.e. they do not hold in general; more details can be found in~\cite{1999:Kloeden:book}.

Next, the moment equations are derived. The CKF technique implies computation of the first two moments~\cite{2010:Haykin}:
\begin{align*}
&\E{x_{k-1}^{(m+1)}}  =\E{f_d\bigl(t_{k-1}^{(m)},x_{k-1}^{(m)}\bigr)} \\
&\var{x_{k-1}^{(m+1)}}\!=\!{\rm \bf var}\Bigl\{\!x_{k-1}^{(m)}\!\!+\delta f\bigl(t_{k-1}^{(m)},x_{k-1}^{(m)}\bigr)\!\! \\
& +\frac{\delta^2}{2} {\mathbb L}_0 f\bigl(t_{k-1}^{(m)},x_{k-1}^{(m)}\bigr)\!\!\Bigr\} \\
& +\frac{\delta^2}{2}\left[G \;{\mathbb L} f^{\top}\!\!\bigl(t_{k-1}^{(m)},x_{k-1}^{(m)}\bigr) +{\mathbb L}  f\bigl(t_{k-1}^{(m)},x_{k-1}^{(m)}\bigr)\;G^{\top}\right]\\
&+\frac{\delta^3}{3}\!\bigl[{\mathbb L} f\bigl(t_{k-1}^{(m)},\hat x_{k-1}^{(m)}\bigr)\!\bigr]\bigl[{\mathbb L} f\bigl(t_{k-1}^{(m)},\hat x_{k-1}^{(m)}\bigr)\!\bigr]^{\top} +\delta G G^{\top}.
\end{align*}

Finally, the third-degree spherical-radial cubature rule is utilized for computing the resulted $n$-dimensional Gaussian-weighted integrals within the CKF approach. For that, the cubature nodes (vectors) are defined~\cite{2009:Haykin,2010:Haykin}:
\begin{align}
\!\!\xi_i  & = \sqrt{n} e_i, \xi_{n+i}=-\sqrt{n} e_i, &\!\!\!{\mathcal X}_{i,k|k}& = S_{P_{k|k}} \xi_i+\hat x_{k|k}  \!\!\label{cub:points}
\end{align}
where $e_i$ ($i=\overline{1,n}$) denotes the $i$-th unit coordinate vector in $\mathbb R^n$ and $n$ is the dimension of the state vector. Next, the quantity $\hat x_{k|k}$ means the state estimate at $t_k$ based on the measurement data $Z_1^k = \{z_1, z_2, \ldots, z_k \}$, i.e. $\hat x_{k|k} = \E{x_{k}|Z_1^k}$, and $\hat x_{k+1|k}$ stands for one-step-ahead predicted estimate. The related error covariance matrices are denoted as follows: $P_{k|k}=\E{ (x_{k}-\hat x_{k|k})(x_{k}-\hat x_{k|k})^{\top}}$ and $P_{k|k-1}=\E{ (x_{k}-\hat x_{k|k-1})(x_{k}-\hat x_{k|k-1})^{\top}}$. For simplicity, the following notations are used throughout the paper. The vector ${\mathbf 1}_{2n}$ is the unitary column of size $2n$, $I_{2n}$ is the identity matrix of size $2n$ and $\otimes$ is the Kronecker tensor product (function \verb"kron" in MATLAB).

The term $S_{P_{k|k}}$ in formula~\eqref{cub:points} stands for the square-root factor of matrix $P_{k|k}$, i.e. $P_{k|k} = S_{P_{k|k}}S_{P_{k|k}}^{\top}$. In the original CKF equations~\cite{2009:Haykin,2010:Haykin}, it is defined in a traditional way by using Cholesky decomposition, i.e. $P_{k|k} = P_{k|k}^{1/2}P_{k|k}^{\top/2}$ where $S_{P_{k|k}}:=P_{k|k}^{1/2}$ is a lower triangular Cholesky factor of $P_{k|k}$. However, an alternative solution can be found under singular value decomposition (SVD). More precisely, one defines $S_{P_{k|k}}:=Q_{P_{k|k}}D_{P_{k|k}}^{1/2}$ where $Q_{P_{k|k}}$ and $D_{P_{k|k}}$ are, respectively, the orthogonal and diagonal SVD factors of a symmetric error covariance matrix $P_{k|k} = Q_{P_{k|k}}D_{P_{k|k}}Q_{P_{k|k}}^{\top}$; see the conventional Algorithms~1a and 2a summarized in Appendix.

As can be seen, Algorithms~1a and 2a utilize SVD for the cubature vectors' generation, but they still propagate the entire error covariance matrices, i.e. they are numerically unstable (with respect to roundoff errors) as all {\it conventional} implementations. For numerical stability reasons, the {\it factored-form} implementations are preferable for practical filters' implementation~\cite[Chapter~7]{2015:Grewal:book}.

\section{SVD factorization-based implementations} \label{main:result}

The new IT-1.5 and EM-0.5 CKF implementations imply SVD of real matrix $A \in {\mathbb R}^{m\times n}$ where $m<n$. Hence, each iteration  of the filtering method includes the rank-$r$ decomposition of matrix $A \in {\mathbb R}^{m\times n}$ that is for $m<n$ and $r<m$ is given by~\cite[Theorem~1.1.6]{2015:Bjorck:book}:
$ A  = W\Sigma V^{\top}, \,
\Sigma =
\begin{bmatrix}
S & 0
\end{bmatrix} \in {\mathbb R}^{m\times n},  \; S={\rm diag}\{ \sigma_1,\ldots,\sigma_r\}
$ where $W \in {\mathbb R}^{m\times m}$, $V \in {\mathbb R}^{n\times n}$ are orthogonal matrices, and $\sigma_1\geq \ldots \geq\sigma_r>0$ are
the singular values of $A$.

It is well known that under the Cholesky-based filtering approach, an orthogonal transformation is used for computing the required Cholesky factorization of a positive definite matrix $C = A^{\top}A+B^{\top}B$. Evidently, the SVD-based solution may not be grounded in the $QR$ factorization utilized within the Cholesky approach because it does not provide the SVD factors requested in the calculation. Therefore, one needs to find a proper way for computing the mentioned SVD factors of the covariance matrix, which has the form $C = A^{\top}A+B^{\top}B$.

For readers' convenience, we utilize the MATLAB built-in function notation \verb"svd" with option \verb"econ" while presenting our new SVD-based CKF methods.  If $m<n$ (this is our case), it performs the reduced SVD variant and returns $W \in {\mathbb R}^{m\times m}$, diagonal $S \in {\mathbb R}^{m\times m}$ and only the first $m$ columns of matrix $V \in {\mathbb R}^{n\times n}$. In fact, the matrix $V$ is of no interest because of a symmetric form and positive definiteness of any covariance matrix that yields only two SVD factors. Note that the output matrices $W$ and $S$ of the same size $m$. If $r = rank(A) < m$, then some diagonal entries of $S \in {\mathbb R}^{m\times m}$ are zero. This can be taken into account for an efficient implementation of the algorithms. Additionally, it paves a way for a reduced-rank filtering strategy. Thus, we formulate the first IT-1.5 CKF estimator within SVD methodology.


\begin{codebox}
\Procname{{\bf Algorithm~1b}. $\proc{IT-1.5 SVD-CKF}$}
\zi \textsc{Initialization:} ($k=0$) Perform $\Pi_0 = Q_{\Pi_0}D_{\Pi_0}Q_{\Pi_0}^{\top}$;
\li \>Set $\hat x_{0|0} = \bar x_0$ and $Q_{P_{0|0}} = Q_{\Pi_0}$, $D^{1/2}_{P_{0|0}} = D^{1/2}_{\Pi_0}$;
\li \>Generate $\xi_i = \sqrt{n} e_i$, $\xi_{n+i}= -\sqrt{n} e_i$, ($i=\overline{1,n}$).
\zi \textsc{Time Update}: ($k=\overline{1,K}$) \Comment{\small\textsc{Priori estimation}}
\li Set $\hat x_{k-1|k-1}^{(0)}\!\!:=\hat x_{k-1|k-1}$, $Q_{P_{k-1|k-1}}^{(0)}\!:=Q_{P_{k-1|k-1}}$
\zi \phantom{Set} and $D_{P_{k-1|k-1}}^{(0),1/2}:=D^{1/2}_{P_{k-1|k-1}}$ at $t_{k-1}$; \label{formuls:start}
\li For $m=0,\ldots,M-1$ do \Comment{\small $t_{k-1}^{(m)} = t_{k-1}\!+m\delta$, $i=\overline{1,2n}$}
\li \>Set factor $S_{P_{k-1|k-1}}^{(m)}\!:=Q^{(m)}_{P_{k-1|k-1}}D^{(m),1/2}_{P_{k-1|k-1}}$;
\li \>Generate ${\mathcal X}^{(m)}_{i,k-1|k-1}\!\!=S_{P_{k-1|k-1}}^{(m)}\xi_i+\hat x^{(m)}_{k-1|k-1}$;
\li \>Propagate ${\mathcal X}_{i,k-1|k-1}^{(m+1)}=f_d\bigl(t_{k-1}^{(m)},{\mathcal X}_{i,k-1|k-1}^{(m)}\bigr)$;
\li \>Collect ${\mathcal X}^{(m+1)}_{k-1|k-1}=\bigl[{\mathcal X}^{(m+1)}_{1,k-1|k-1},\ldots,{\mathcal X}^{(m+1)}_{2n,k-1|k-1}\bigr]$;
\li \>Find the estimate $\hat x_{k-1|k-1}^{(m+1)}=\frac{1}{2n} {\mathcal X}^{(m+1)}_{k-1|k-1} {\mathbf 1}_{2n}$;
\li \>Set ${\mathbb X}^{(m+1)}_{k-1|k-1}=\frac{1}{\sqrt{2n}}\bigl({\mathcal X}^{(m+1)}_{k-1|k-1}\!-{\mathbf
1}_{2n}^{\top}\otimes \hat x_{k-1|k-1}^{(m+1)}\bigr)$;
\li \>Find ${\mathbb L}f_{k-1}^{(m)} = {\mathbb L}f(t_{k-1}^{(m)},\hat x_{k-1|k-1}^{(m)})$ and collect \label{svd:it:p1:predict}
\zi \>$T = \left[{\mathbb X}^{(m+1)}_{k-1|k-1}, \;\; \sqrt{\delta}(G+\frac{\delta}{2}{\mathbb L}f_{k-1}^{(m)}),
\; \; \sqrt{\frac{\delta^3}{12}}{\mathbb L}f_{k-1}^{(m)}\right]$;
\li \>Perform $[Q_{P_{k-1|k-1}}^{(m+1)}, D^{(m+1),1/2}_{P_{k-1|k-1}}] \leftarrow \verb"svd"(T,econ)$; \label{svd:it:p2:predict}
\li Set $\hat x_{k|k-1}\!:=\hat x_{k-1|k-1}^{(M)}$, $Q_{P_{k|k-1}}\!:=Q_{P_{k-1|k-1}}^{(M)}$
\zi \phantom{Set} and the diagonal part $D^{1/2}_{P_{k|k-1}}\!:=D_{P_{k-1|k-1}}^{(M),1/2}$.
\zi \textsc{Measurement Update}: \Comment{\small\textsc{Posteriori
estimation}}
\li \>Set square-root factor $S_{P_{k|k-1}}\!:=\!Q_{P_{k|k-1}}D_{P_{k|k-1}}^{1/2}$; \label{svd:it:MU:start}
\li \>Generate ${\mathcal X}_{i,k|k-1}=S_{P_{k|k-1}}\xi_i+\hat x_{k|k-1}$;
\li \>Propagate ${\mathcal Z}_{i,k|k-1}=h\bigl(k,{\mathcal X}_{i,k|k-1}\bigr)$;
\li \>Collect matrix ${\mathcal Z}_{k|k-1}=\bigl[{\mathcal
Z}_{1,k|k-1},\ldots,{\mathcal Z}_{2n,k|k-1} \bigr]$;
\li \>Compute  $\hat z_{k|k-1}=\frac{1}{2n}{\mathcal Z}_{k|k-1} {\mathbf 1}_{2n}$;
\li \>Collect matrix ${\mathcal X}_{k|k-1}=\bigl[{\mathcal
X}_{1,k|k-1},\ldots,{\mathcal X}_{2n,k|k-1}\bigr]$;
\li \>Set ${\mathbb X}_{k|k-1}=\frac{1}{\sqrt{2n}}\bigl({\mathcal X}_{k|k-1}\!-{\mathbf
1}_{2n}^{\top}\otimes \hat x_{k|k-1}\bigr)$;
\li \>Set ${\mathbb Z}_{k|k-1}=\frac{1}{\sqrt{2n}}\bigl({\mathcal Z}_{k|k-1}\!-{\mathbf
1}_{2n}^{\top}\otimes \hat z_{k|k-1}\bigr)$;
\li \>From pre-array $A = \left[{\mathbb Z}_{k|k-1}, \;\; Q_{R_k}D_{R_k}^{1/2}\right]$; \label{svd:it:rek1}
\li \>Read-off factors $[Q_{R_{e,k}}, D^{1/2}_{R_{e,k}}] \leftarrow \verb"svd"(A,econ)$; \label{svd:it:rek2}
\li \>Compute cross-covariance $P_{xz,k}={\mathbb X}_{k|k-1}{\mathbb Z}_{k|k-1}^{\top}$;
\li \>Find cubature gain ${\mathbb K}_{k}=P_{xz,k}Q_{R_{e,k}}D_{R_{e,k}}^{-1}Q_{R_{e,k}}^{\top}$; \label{svd:it:gain}
\li \>Update estimate $\hat x_{k|k}=\hat x_{k|k-1}+{\mathbb K}_k(z_k-\hat z_{k|k-1})$;
\li \>Collect $B = \left[{\mathbb X}_{k|k-1} - {\mathbb K}_{k}{\mathbb Z}_{k|k-1}, \;\; {\mathbb K}_{k}Q_{R_k}D_{R_k}^{1/2}\right]$; \label{svd:it:f:p1}
\li \>Read-off factors $[Q_{P_{k|k}}, D^{1/2}_{P_{k|k}}] \leftarrow \verb"svd"(B,econ)$; \label{svd:it:f:p2}
\end{codebox}

We observe that the filter covariances are SVD-factorized in Algorithm~1b once. Next, the involved SVD factors are recursively propagated and updated instead of entire matrices $P_{k|k-1}$ and $P_{k|k}$. Thus, Algorithm~1b is of {\it factored-form} (SR) form.

\begin{Prop} \label{proposition:1}
The SVD-based IT-1.5 CKF equations in Algorithm~1b are algebraically equivalent to the conventional  IT-1.5 CKF formulas in Algorithm~1a.
\end{Prop}

\begin{proof} Consider the equation in line~\ref{con:it:p1:predict} of Algorithm~1a for computing $P_{k-1|k-1}^{(m+1)}$. The aim is to derive its alternative variant that recursively propagate the SVD factors of $P_{k-1|k-1}^{(m+1)}$. For that, we assume that the SVD factors at the current step  $Q_{P_{k-1|k-1}}^{(m)}$ and $D^{(m),1/2}_{P_{k-1|k-1}}$ are known. We need to find a way to propagate them according to the discussed equation, which we re-arrange
\begin{align}
& P_{k-1|k-1}^{(m+1)}  = \left[{\mathbb X}^{(m+1)}_{k-1|k-1}, \; \sqrt{\delta}(G+\frac{\delta}{2}{\mathbb L}f_{k-1}^{(m)}), \; \sqrt{\frac{\delta^3}{12}}{\mathbb L}f_{k-1}^{(m)}\right] \nonumber \\
& \times \left[{\mathbb X}^{(m+1)}_{k-1|k-1}, \; \sqrt{\delta}(G+\frac{\delta}{2}{\mathbb L}f_{k-1}^{(m)}), \; \sqrt{\frac{\delta^3}{12}}{\mathbb L}f_{k-1}^{(m)}\right]^{\top} \label{eq:proof:1}
\end{align}
where ${\mathbb L}f_{k-1}^{(m)} = {\mathbb L}f(t_{k-1}^{(m)},\hat x_{k-1|k-1}^{(m)})$. Denote the generalized SR matrix factor of $P_{k-1|k-1}^{(m+1)}$ appeared in~\eqref{eq:proof:1} as
\begin{equation}
T:= \left[{\mathbb X}^{(m+1)}_{k-1|k-1}, \sqrt{\delta}(G+\frac{\delta}{2}{\mathbb L}f_{k-1}^{(m)}), \; \sqrt{\frac{\delta^3}{12}}{\mathbb L}f_{k-1}^{(m)}\right], \label{eq:proof:2}
\end{equation}
i.e. $P_{k-1|k-1}^{(m+1)}=TT^{\top}$ is a generalized SR matrix representation. Our goal is to transform the generalized SR factor $T$ in~\eqref{eq:proof:2} into the SVD-based SR factors, i.e. into the orthogonal factors
$W$, $V$ and the diagonal factor $S$. In order to do so, SVD should be used, i.e. $T=WSV^{\top}$. Thus, we have the equality between two SR-type representations for $P_{k-1|k-1}^{(m+1)}=TT^{\top}$ that is
\begin{equation}
TT^{\top} = WSV^{\top}(WSV^{\top})^{\top} = WS^2W^{\top} \label{eq:proof:3}
\end{equation}
where $T$ is defined in~\eqref{eq:proof:2} through known previous values $Q_{P_{k-1|k-1}}^{(m)}$ and $D^{(m),1/2}_{P_{k-1|k-1}}$, but the orthogonal $W$ and diagonal $S$ should be determined. Recall $TT^{\top} = P_{k-1|k-1}^{(m+1)}$ and, hence, from equality~\eqref{eq:proof:3} we conclude that $W$ is, in fact, the orthogonal SVD factor of  $P_{k-1|k-1}^{(m+1)}$, i.e. $W:=Q_{P_{k-1|k-1}}^{(m+1)}$. Similarly, we define $S:=D^{(m+1),1/2}_{P_{k-1|k-1}}$. This yields the required recursion where the propagated SVD-based SR factors $Q_{P_{k-1|k-1}}^{(m+1)}$ and $D^{(m+1),1/2}_{P_{k-1|k-1}}$ are simply read-off from the resulted post-arrays after SVD in line~\ref{svd:it:p2:predict} of Algorithm~1b. This justifies the time update.

Similarly, we represent the residual covariance equation appeared in Algorithm~1a in a symmetric SR form, i.e.
\begin{align}
R_{e,k} & = {\mathbb Z}_{k|k-1}{\mathbb Z}_{k|k-1}^{\top}+R_k \nonumber \\
& = \left[{\mathbb Z}_{k|k-1}, Q_{R_k}D_{R_k}^{1/2}\right]\left[{\mathbb Z}_{k|k-1}, Q_{R_k}D_{R_k}^{1/2}\right]^{\top}.
\label{eq:proof:4}
\end{align}

Again, our goal is to get the SVD-based SR factors $Q_{R_{e,k}}$ and $D_{R_{e,k}}$ instead of the generalized SR factor
\begin{equation}
A:=\left[{\mathbb Z}_{k|k-1}, \quad Q_{R_k}D_{R_k}^{1/2}\right].
\label{eq:proof:5}
\end{equation}
Having applied the SVD factorization to the pre-array $A$ in~\eqref{eq:proof:5} and taking into account the equality between two SR representations in~\eqref{eq:proof:3}, we conclude that $W:=Q_{R_{e,k}}$ and $S:=D^{1/2}_{R_{e,k}}$. This gives us the computational scheme summarized in lines~\ref{svd:it:rek1}, \ref{svd:it:rek2} of Algorithm~1b for calculating the SVD factors  $Q_{R_{e,k}}$ and $D_{R_{e,k}}$ instead of the full matrix  $R_{e,k}$ utilized in the conventional Algorithm~1a.

Having computed the SVD factors of $R_{e,k}$ and taking into account the properties of orthogonal matrices, the cubature gain is calculated in a more robust way as
\[
{\mathbb K}_{k} =P_{xz,k}R_{e,k}^{-1} = P_{xz,k}Q_{R_{e,k}}D_{R_{e,k}}^{-1}Q_{R_{e,k}}^{\top}.
\]
Indeed, the full matrix inversion operation $R_{e,k}^{-1}$ required in the conventional Algorithm~1a is replaced by the diagonal matrix inversion only, i.e. $D_{R_{e,k}}^{-1}$. This explains the improved numerical robustness of the new SVD-based CKF methods, which we discuss in the next section.

Finally, we derive a symmetric equation for $P_{k|k}$. From ${\mathbb K}_{k}=P_{xz,k}R_{e,k}^{-1}$, we get
${\mathbb K}_{k}P_{xz,k}^{\top}={\mathbb K}_{k}R_{e,k}{\mathbb K}_{k}^{\top}$ and
\begin{equation}
P_{k|k}= P_{k|k-1} - {\mathbb K}_k R_{e,k} {\mathbb K}_k^{\top}  = P_{k|k-1} - P_{xz,k}{\mathbb K}_k^{\top}. \label{eq:proof1}
\end{equation}
From $P_{xz,k}={\mathbb X}_{k|k-1}{\mathbb Z}_{k|k-1}^{\top}$ and $P_{k|k-1} = {\mathbb X}_{k|k-1}{\mathbb X}_{k|k-1}^{\top}$, after
summation~\eqref{eq:proof1} with ${\mathbb K}_{k}P_{xz,k}^{\top}={\mathbb K}_{k}R_{e,k}{\mathbb K}_{k}^{\top}$, get
\begin{align*}
&P_{k|k} = P_{k|k-1} - P_{xz,k}{\mathbb K}_k^{\top} + {\mathbb K}_{k}R_{e,k}{\mathbb K}_{k}^{\top}  - {\mathbb K}_{k}P_{xz,k}^{\top} \\
& =  {\mathbb X}_{k|k-1}{\mathbb X}_{k|k-1}^{\top} - {\mathbb X}_{k|k-1}{\mathbb Z}_{k|k-1}^{\top}{\mathbb K}_k^{\top} - {\mathbb K}_{k}{\mathbb Z}_{k|k-1}{\mathbb X}^{\top}_{k|k-1} \\
& \phantom{=} + {\mathbb K}_{k}({\mathbb Z}_{k|k-1}{\mathbb Z}_{k|k-1}^{\top}+R_k){\mathbb K}_{k}^{\top}  = {\mathbb K}_{k}R_k {\mathbb K}_{k}^{\top} \\
& \phantom{=} + \left[{\mathbb X}_{k|k-1}-{\mathbb K}_{k}{\mathbb Z}_{k|k-1}\right]\left[{\mathbb X}_{k|k-1}-{\mathbb K}_{k}{\mathbb Z}_{k|k-1}\right]^{\top}.
\end{align*}
The resulted symmetric formula allows for the generalized SR factorization of the error covariance $P_{k|k}$ as
\begin{align*}
BB^{\top} & = \left[{\mathbb X}_{k|k-1} - {\mathbb K}_{k}{\mathbb Z}_{k|k-1}, \; {\mathbb K}_{k}Q_{R_k}D_{R_k}^{1/2}\right]\\
 & \times \left[{\mathbb X}_{k|k-1} - {\mathbb K}_{k}{\mathbb Z}_{k|k-1}, \; {\mathbb K}_{k}Q_{R_k}D_{R_k}^{1/2}\right]^{\top} = P_{k|k}.
\end{align*}
Similarly to our previous derivation, we SVD factorize the pre-array $B$ in order to obtain the related SVD factors of the filter covariance
$Q_{P_{k|k}}$ and $D_{P_{k|k}}$. This justifies the computations in lines~\ref{svd:it:f:p1}, \ref{svd:it:f:p2} of Algorithm~1b and completes the proof.
\end{proof}


\begin{codebox}
\Procname{{\bf Algorithm~2b}. $\proc{EM-0.5 SVD-CKF}$}
\zi \textsc{Initialization:} ($k=0$) Repeat Algorithm~1b.
\zi \textsc{Time Update}: ($k=\overline{1,K}$) \Comment{\small\textsc{Priori estimation}}
\li Set $\hat x_{k-1|k-1}^{(0)}\!\!:=\hat x_{k-1|k-1}$, $Q_{P_{k-1|k-1}}^{(0)}\!\!:=Q_{P_{k-1|k-1}}$
\zi \phantom{Set} and $D_{P_{k-1|k-1}}^{(0),1/2}:=D^{1/2}_{P_{k-1|k-1}}$ at $t_{k-1}$;
\li For $m=0,\ldots,M-1$ do \Comment{\small $t_{k-1}^{(m)} = t_{k-1}\!+m\delta$}
\li \>Set $S_{P_{k-1|k-1}}^{(m)}\!:=Q^{(m)}_{P_{k-1|k-1}}D^{(m),1/2}_{P_{k-1|k-1}}$;
\li \>Generate ${\mathcal X}^{(m)}_{i,k-1|k-1}\!\!=S_{P_{k-1|k-1}}^{(m)}\xi_i+\hat x^{(m)}_{k-1|k-1}$;
\li \>${\mathcal X}_{i,k-1|k-1}^{(m+1)}={\mathcal X}_{i,k-1|k-1}^{(m)} + \delta f\bigl(t_{k-1}^{(m)},{\mathcal X}_{i,k-1|k-1}^{(m)}\bigr)$;
\li \>Set ${\mathcal X}^{(m+1)}_{k-1|k-1}=\bigl[{\mathcal X}^{(m+1)}_{1,k-1|k-1},\ldots,{\mathcal X}^{(m+1)}_{2n,k-1|k-1}\bigr]$;
\li \>Compute $\hat x_{k-1|k-1}^{(m+1)}=\frac{1}{2n} {\mathcal X}^{(m+1)}_{k-1|k-1} {\mathbf 1}_{2n}$;
\li \>${\mathbb X}^{(m+1)}_{k-1|k-1}=\frac{1}{\sqrt{2n}}\bigl({\mathcal X}^{(m+1)}_{k-1|k-1}\!\!-{\mathbf
1}_{2n}^{\top}\!\otimes \hat x_{k-1|k-1}^{(m+1)}\bigr)$;
\li \>Pre-array $T = \left[{\mathbb X}^{(m+1)}_{k-1|k-1}, \;\; \sqrt{\delta}\tilde GQ_{Q}D^{1/2}_{Q}\right]$; \label{svd:em:p1:predict}
\li \>$[Q_{P_{k-1|k-1}}^{(m+1)}, D^{(m+1),1/2}_{P_{k-1|k-1}}]\leftarrow \verb"svd"(T,econ)$; \label{svd:em:p2:predict}
\li Set $\hat x_{k|k-1}\!:=\hat x_{k-1|k-1}^{(M)}$, $Q_{P_{k|k-1}}\!:=Q_{P_{k-1|k-1}}^{(M)}$
\zi \phantom{Set} and the diagonal part $D^{1/2}_{P_{k|k-1}}\!:=D_{P_{k-1|k-1}}^{(M),1/2}$.
\zi \textsc{Measurement Update}: Repeat lines~\ref{svd:it:MU:start}-\ref{svd:it:f:p2}
\zi \>of Algorithm~1b (the IT-1.5 SVD-CKF).
\end{codebox}


\begin{Prop} \label{proposition:2}
The SVD-based EM-0.5 CKF equations in Algorithm~2b are algebraically equivalent to the conventional EM-0.5 CKF formulas in Algorithm~2a.
\end{Prop}

\begin{proof} Lines~\ref{svd:em:p1:predict}, \ref{svd:em:p2:predict} of Algorithm~2b are justified by factorizing the equation in line~\ref{con:em:p1:predict} of Algorithm~2a as
\begin{align*}
P_{k-1|k-1}^{(m+1)} & = \left[{\mathbb X}^{(m+1)}_{k-1|k-1}, \sqrt{\delta}\tilde GQ_{Q}D^{1/2}_{Q}\right] \\
& \times \left[{\mathbb X}^{(m+1)}_{k-1|k-1}, \sqrt{\delta}\tilde GQ_{Q}D^{1/2}_{Q}\right]^{\top}.
\end{align*}
This completes the proof.
\end{proof}

\begin{table*}
\caption{The $\mbox{\rm ARMSE}_p$ observed in the target tracking scenario with  the increasingly ill-conditioned measurement scheme} \label{Tab:1}
\centering
\begin{tabular}{c||c|c|c|c||c|c|c|c}
\hline
&  \multicolumn{4}{c||}{IT-1.5 discretization-based CKFs ($64$ subintervals)} &  \multicolumn{4}{c}{EM-0.5 discretization-based CKFs ($512$ subintervals)} \\
\cline{2-9}
$\delta$ & \texttt{original}~\cite{2010:Haykin} & \texttt{Alg.1a} & \texttt{Cholesky}~\cite{2010:Haykin} &  \texttt{SVD Alg.1b} & \texttt{standard} & \texttt{Alg.2a} & \texttt{Cholesky}~\cite{2018:Kulikov:Romania}  & \texttt{SVD Alg.2b} \\
\hline
1.0e-01 & 6.014e+00 & 6.018e+00 & 6.014e+00 & 6.018e+00 & 1.325e+02 & 1.325e+02 & 1.325e+02 & 1.325e+02 \\
1.0e-02 & 6.011e+00 & 6.011e+00 & 6.011e+00 & 6.011e+00 & 1.324e+02 & 1.324e+02 & 1.324e+02 & 1.324e+02 \\
1.0e-03 & 6.010e+00 & 6.009e+00 & 6.010e+00 & 6.009e+00 & 1.324e+02 & 1.324e+02 & 1.324e+02 & 1.324e+02 \\
1.0e-04 & ---       & 6.009e+00 & 6.010e+00 & 6.009e+00 & 1.323e+02 & 1.323e+02 & 1.323e+02 & 1.323e+02 \\
1.0e-05 & ---       & 6.009e+00 & 6.010e+00 & 6.009e+00 & ---       & 1.323e+02 & 1.323e+02 & 1.323e+02 \\
1.0e-06 & ---       & 6.009e+00 & 6.010e+00 & 6.009e+00 & ---       & 1.323e+02 & 1.323e+02 & 1.323e+02 \\
1.0e-07 & ---       & ---       & 6.010e+00 & 6.009e+00 & ---       & ---       & 1.323e+02 & 1.323e+02 \\
1.0e-08 & ---       & ---       & 6.010e+00 & 6.009e+00 & ---       & ---       & 1.323e+02 & 1.323e+02 \\
1.0e-09 & ---       & ---       & 6.010e+00 & 6.009e+00 & ---       & ---       & 1.323e+02 & 1.323e+02 \\
1.0e-10 & ---       & ---       & 6.010e+00 & 6.009e+00 & ---       & ---       & 1.324e+02 & 1.324e+02 \\
1.0e-11 & ---       & ---       & 6.020e+00 & 6.021e+00 & ---       & ---       & 1.323e+02 & 1.323e+02 \\
1.0e-12 & ---       & ---       & 7.925e+00 & 7.926e+00 & ---       & ---       & 1.313e+02 & 1.313e+02 \\
1.0e-13 & ---       & ---       & 8.335e+00 & 8.334e+00 & ---       & ---       & 1.219e+02 & 1.220e+02 \\
1.0e-14 & ---       & ---       & 1.111e+01 & 1.128e+01 & ---       & ---       & 1.188e+02 & 1.182e+02 \\
\hline
\end{tabular}
\end{table*}

\begin{remark}
If stochastic system~\eqref{eq1.1}, \eqref{eq1.2} is parameterized, then the recent theoretical result in~\cite[Lemma~1]{2017:Tsyganova:IEEE} applied to Algorithms~1b and~2b yields their SVD-based extended ``differentiated'' version, which can be used for simultaneous hidden dynamic state and unknown system parameters estimation by gradient-based CKF adaptive schemes. A similar adaptive CKF method within the Cholesky approach has been proposed in~\cite{2018:Boureghda}.
\end{remark}

\section{Numerical Experiments} \label{numerical:experiments}

For a comparative study of the newly-devised SVD-based CKF methods with the previously published conventional and Cholesky-based CKF algorithms, we consider the radar tracking problem from~\cite[Sec.~VIII]{2010:Haykin} but with incorporated ill-conditioned measurement scheme artificially simulated as discussed in~\cite[Examples 7.1]{1969:Dyer} and~\cite[Section~7.2.2]{2015:Grewal:book} for provoking numerical instabilities due to roundoff in the state-space model at hand.

\begin{exmp} When performing a coordinated turn, the aircraft's dynamics obeys~\eqref{eq1.1} with the drift function
$f(\cdot)=\left[\dot{\epsilon}, -\omega \dot{\eta}, \dot{\eta}, \omega \dot{\epsilon}, \dot{\zeta},  0, 0\right]$ and the {\it standard} Brownian motion, i.e. $Q=I$ with $G={\rm diag}\left[0,\sigma_1,0,\sigma_1,0,\sigma_1,\sigma_2\right]$, $\sigma_1=\sqrt{0.2}$, $\sigma_2=0.007$.
The state consists of three positions, corresponding velocities and the turn rate, that is $x(t)= [\epsilon, \dot{\epsilon}, \eta, \dot{\eta}, \zeta, \dot{\zeta}, \omega]^{\top}$. The initial values are $\bar x_0=[1000\,\mbox{\rm m}, 0\,\mbox{\rm m/s}, 2650\,\mbox{\rm m},150\,\mbox{\rm m/s}, 200\,\mbox{\rm m}, 0\,\mbox{\rm m/s},\omega^\circ/\mbox{\rm s}]^{\top}$ and $\Pi_0=\mbox{\rm diag}(0.01\,I_7)$. The state is observed through
\begin{align*}
z_k & =
\begin{bmatrix}
1 & 1 & 1 & 1 & 1 &  1 &  1\\
1 & 1 & 1 & 1 & 1 &  1 &  1 +\delta
\end{bmatrix}
x_k +
\begin{bmatrix}
v_k^1 \\
v_k^2
\end{bmatrix}, \; R_k=\delta^{2}I_2
\end{align*}
where parameter $\delta$ is used for simulating roundoff effect.
This increasingly ill-conditioned target tracking scenario assumes that $\delta\to 0$, i.e. $\delta=10^{-1},10^{-2},\ldots,10^{-14}$.
\end{exmp}

When the ill-conditioning parameter $\delta$ tends to a machine precision limit, i.e. $\delta \to \epsilon_{roundoff}$, the measurement noise covariance $R_k$ becomes an ill-conditioned matrix and it is, in fact, close to zero, which implies the scenario of `almost exact' or `perfect' measurements.
A singular noise covariance does not, in general, present theoretical problems. However,  from a numerical point of view, this situation is highly unpleasant because the singular noise covariance increases the possibility of numerical problems such as the ill-conditioning of the
innovations covariance; see~\cite[p.~249]{1979:Maybeck:book}, \cite[p.~365]{1994:Stengel:book}, \cite[p.~1305]{2010:Simon}. To get some insights of the internal numerical calculations involved in the discussed state estimation scenario, we follow the error model derivation in~\cite{1986:Verhaegen}. In the cited paper, it is proved that the key parameter, which influences the error propagation in the {\it covariance}-form filtering algorithms, is the condition number of the innovations covariance matrix $R_{e,k}$. Example~1 is designed in such a way that although the matrix $H_k$ is of full rank for any value of $\delta$ in use, the matrix $R_{e,k}$ becomes ill-conditioned when $\delta \to \epsilon_{roundoff}$ and $R_k$ tends to a singular (and even zero) matrix~\cite[Section~7.2.2]{2015:Grewal:book}. In other words, Example~1 illustrates how the given well-posed problem is made to be ill-conditioned one by the filter implementation due to roundoff errors. More precisely, the accumulated roundoff error, which is negligible in well-posed state estimation scenarios, may blow up and affect severely the accuracy and performance of the filters applied in ill-conditioned situations.

In our numerical experiments, we solve a filtering problem on interval $[0s, 150s]$ with sampling period $\Delta = \Delta_k = 1s$ by eight CKF implementations listed in Table~1 for various ill-conditioned scenarios. The turn rate is set to $\omega=3^\circ/\mbox{\rm s}$. All filtering methods are tested at the same conditions, i.e. with the same simulated `true' state trajectory, the same measurement data and the same initial conditions. We perform $100$ Monte Carlo simulations in order to compute the root mean square error in position $\mbox{\rm ARMSE}_p$; see the formulas in~\cite{2010:Haykin}. Finally, in contrast to~\cite{2010:Haykin}, we test the EM-0.5 CKF implementations and IT-1.5 CKF methods separately, because their integrated numerical schemes yield different discretization errors. Thus, the first panel in Table~1 summarizes the IT-1.5 discretization-based CKF results while the second panel contains the EM-0.5-based filters' outcomes.

Having analyzed the data in Table~1, we make a few important conclusions. First, when the problem is well-conditioned, i.e. when $\delta$ is large, we observe the identical performance of the conventional implementations and their related square-root variants. This is a consequence of the algebraic equivalence between the SVD-based Algorithms~1b,~2b and their conventional variants in Algorithms~1a,~2a proved in this paper. The same conclusion holds for the original CKF implementation and its Cholesky-based method designed in~\cite{2010:Haykin}. We emphasize that despite the estimation error in the EM-0.5 CKF methods is higher than that in the IT-1.5 CKF algorithms all these work with no failure in our increasingly ill-conditioned target tracking scenario.

Next, having compared the first column to the second one in each panel of Table~1, we observe the difference of two {\it conventional} CKF implementations with different ways of the cubature node generation. As can be seen, the original CKF grounded in the Cholesky-based generation is the most vulnerable algorithm in Example~1. More precisely, it will produce the decent state estimates if the target tracking scenario at hand is rather well-conditioned, i.e. when $\delta \ge 10^{-3}$, but it fails at $\delta=10^{-4}$. In this case, the roundoff errors destroy the positivity of the filtering covariance matrix, that fails immediately the Cholesky decomposition. On the other hand, its replacement with the SVD-type factorization improves obviously the reliability of the conventional CKF. Indeed, Algorithms~1a and~2a are both successful and accurate when  $\delta \ge 10^{-6}$, but these do not suit for treating severely ill-conditioned problems and fail when at $\delta=10^{-7}$.

Eventually, we conclude that (i) the new square-root CKF algorithms derived under the SVD methodology outperform their related conventional CKF variants, and (ii) the numerical stability of two approaches, which are Cholesky- and SVD-based  methods, are similar because these are both square-root solutions and, hence, enjoy the similar numerical robustness to roundoff. We observe that their divergence speed due to roundoff is the same in the ill-conditioned stochastic system in use and these are more numerically stable in comparison to the conventional CKF. On the other hand, the SVD-based CKF methods can be preferable in practice because of their extended functionality and useful features mentioned in Section~\ref{introduction}. Thus, the newly-suggested SVD-based CKF implementation approach provides practitioners with a diversity of methods giving a fair possibility for choosing any of them depending on a real-world application at hand, its complexity and accuracy requirements.

\section*{Acknowledgments}
The authors acknowledge the financial support of the Portuguese FCT~--- \emph{Funda\c{c}\~ao para a Ci\^encia e a Tecnologia}, through the projects UIDB/04621/2020 and UIDP/04621/2020 of CEMAT/IST-ID, Center for Computational and Stochastic Mathematics, Instituto Superior T\'ecnico, University of Lisbon.

\appendix{\bf Appendix. Conventional CKF implementations}

\begin{codebox}
\Procname{{\bf Algorithm~1a}. $\proc{IT-1.5 CKF}$}
\zi \textsc{Initialization:} ($k=0$) Set $\hat x_{0|0} = \bar x_0$, $P_{0|0} = \Pi_0$;
\li \>Generate $\xi_i = \sqrt{n} e_i$, $\xi_{n+i}= -\sqrt{n} e_i$, ($i=\overline{1,n}$).
\zi \textsc{Time Update}: ($k=\overline{1,K}$) \Comment{\small\textsc{Priori estimation}}
\li Set $\hat x_{k-1|k-1}^{(0)}:=\hat x_{k-1|k-1}$ and $P_{k-1|k-1}^{(0)}:=P_{k-1|k-1}$;  \label{formuls:start}
\li For $m=0,\ldots,M-1$ do \Comment{\small $t_{k-1}^{(m)} = t_{k-1}\!+m\delta$, $[t_{k-1}, t_k]$}
\li \>SVD: $P_{k-1|k-1}^{(m)}= Q^{(m)}_{P_{k-1|k-1}}\!\!\!\!D^{(m)}_{P_{k-1|k-1}}[Q^{(m)}_{P_{k-1|k-1}}]^{\top}$;
\li \>Set $S_{P_{k-1|k-1}}^{(m)}\!\!:=\!Q^{(m)}_{P_{k-1|k-1}}\!\!\!\![D^{(m)}_{P_{k-1|k-1}}]^{1/2}$;
\li \>${\mathcal X}^{(m)}_{i,k-1|k-1}\!\!=S_{P_{k-1|k-1}}^{(m)}\!\xi_i+\hat x^{(m)}_{k-1|k-1}$, ($i=\overline{1,2n}$);
\li \>${\mathcal X}_{i,k-1|k-1}^{(m+1)}=f_d\bigl(t_{k-1}^{(m)},{\mathcal X}_{i,k-1|k-1}^{(m)}\bigr)$, ($i=\overline{1,2n}$);
\li \>Collect ${\mathcal X}^{(m+1)}_{k-1|k-1}=\bigl[{\mathcal X}^{(m+1)}_{1,k-1|k-1},\ldots,{\mathcal X}^{(m+1)}_{2n,k-1|k-1}\bigr]$;
\li \>Find the estimate $\hat x_{k-1|k-1}^{(m+1)}=\frac{1}{2n} {\mathcal X}^{(m+1)}_{k-1|k-1} {\mathbf 1}_{2n}$;
\li \>Set ${\mathbb X}^{(m+1)}_{k-1|k-1}=\frac{1}{\sqrt{2n}}\bigl({\mathcal X}^{(m+1)}_{k-1|k-1}\!\!-{\mathbf
1}_{2n}^{\top}\!\otimes \hat x_{k-1|k-1}^{(m+1)}\bigr)$;
\li \>Find $P_{k-1|k-1}^{(m+1)} = {\mathbb X}^{(m+1)}_{k-1|k-1} \left[{\mathbb X}^{(m+1)}_{k-1|k-1}\right]^{\top}$ \label{con:it:p1:predict}
\zi \>       $\quad +\frac{\delta^2}{2}\! \!\left[G{\mathbb L}f^{\top}\bigl(t_{k-1}^{(m)},\hat x_{k-1|k-1}^{(m)}\bigr)\right.$
\zi \>       $\quad \left.+{\mathbb L} f\bigl(t_{k-1}^{(m)},\hat x_{k-1|k-1}^{(m)}\bigr)G^{\top}\right]  + \delta G G^{\top}$
\zi \>       $\quad +\frac{\delta^3}{3}\!\bigl[{\mathbb L} f\bigl(t_{k-1}^{(m)},\hat x_{k-1|k-1}^{(m)}\bigr)\!\bigr]\bigl[{\mathbb L} f\bigl(t_{k-1}^{(m)},
\hat x_{k-1|k-1}^{(m)}\bigr)\!\bigr]^{\top}$;
\li Set $\hat x_{k|k-1}: =\hat x_{k-1|k-1}^{(M)}$ and $P_{k|k-1}:=P_{k-1|k-1}^{(M)}$.
\zi \textsc{Measurement Update}: \Comment{\small\textsc{Posteriori
estimation}}
\li \>Apply SVD: $P_{k|k-1}=Q_{P_{k|k-1}}D_{P_{k|k-1}}Q_{P_{k|k-1}}^{\top}$; \label{con:it:MU:start}
\li \>Set $S_{P_{k|k-1}}\!:=\!Q_{P_{k|k-1}}D_{P_{k|k-1}}^{1/2}$;
\li \>Get ${\mathcal X}_{i,k|k-1}=S_{P_{k|k-1}}\xi_i+\hat x_{k|k-1}$, ($i=\overline{1,2n}$);
\li \>Propagate ${\mathcal Z}_{i,k|k-1}=h\bigl(k,{\mathcal X}_{i,k|k-1}\bigr)$, ($i=\overline{1,2n}$);
\li \>Collect  ${\mathcal Z}_{k|k-1}=\bigl[{\mathcal
Z}_{1,k|k-1},\ldots,{\mathcal Z}_{2n,k|k-1} \bigr]$;
\li \>Compute the predicted $\hat z_{k|k-1}=\frac{1}{2n}{\mathcal Z}_{k|k-1} {\mathbf 1}_{2n}$;
\li \>Collect ${\mathcal X}_{k|k-1}=\bigl[{\mathcal
X}_{1,k|k-1},\ldots,{\mathcal X}_{2n,k|k-1}\bigr]$;
\li \>Set ${\mathbb X}_{k|k-1}=\frac{1}{\sqrt{2n}}\bigl({\mathcal X}_{k|k-1}\!-{\mathbf
1}_{2n}^{\top}\otimes \hat x_{k|k-1}\bigr)$;
\li \>Set ${\mathbb Z}_{k|k-1}=\frac{1}{\sqrt{2n}}\bigl({\mathcal Z}_{k|k-1}\!-{\mathbf
1}_{2n}^{\top}\otimes \hat z_{k|k-1}\bigr)$;
\li \>Compute $R_{e,k}={\mathbb Z}_{k|k-1}{\mathbb Z}_{k|k-1}^{\top}+R_k$; \label{con:it:rek}
\li \>Find cross-covariance $P_{xz,k}={\mathbb X}_{k|k-1}{\mathbb Z}_{k|k-1}^{\top}$;
\li \>Calculate cubature gain ${\mathbb K}_{k}=P_{xz,k}R_{e,k}^{-1}$; \label{con:it:gain}
\li \>Update $\hat x_{k|k}=\hat x_{k|k-1}+{\mathbb K}_k(z_k-\hat z_{k|k-1})$;
\li \>Update $P_{k|k}=P_{k|k-1} - {\mathbb K}_k R_{e,k} {\mathbb K}_k^{\top}$. \label{con:it:f:p}
\end{codebox}

\begin{codebox}
\Procname{{\bf Algorithm~2a}. $\proc{EM-0.5 CKF}$}
\zi \textsc{Initialization:} ($k=0$) Set $\hat x_{0|0} = \bar x_0$, $P_{0|0} = \Pi_0$;
\li \>Generate $\xi_i = \sqrt{n} e_i$, $\xi_{n+i}= -\sqrt{n} e_i$, ($i=\overline{1,n}$).
\zi \textsc{Time Update}: ($k=\overline{1,K}$) \Comment{\small\textsc{Priori estimation}}
\li Set $\hat x_{k-1|k-1}^{(0)}:=\hat x_{k-1|k-1}$ and $P_{k-1|k-1}^{(0)}:=P_{k-1|k-1}$;
\li For $m=0,\ldots,M-1$ do \Comment{\small $t_{k-1}^{(m)} = t_{k-1}\!+m\delta$, $[t_{k-1}, t_k]$}
\li \>SVD: $P_{k-1|k-1}^{(m)}= Q^{(m)}_{P_{k-1|k-1}}\!\!\!\!D^{(m)}_{P_{k-1|k-1}}[Q^{(m)}_{P_{k-1|k-1}}]^{\top}$;
\li \>Set $S_{P_{k-1|k-1}}^{(m)}\!\!:=\!Q^{(m)}_{P_{k-1|k-1}}\!\!\!\![D^{(m)}_{P_{k-1|k-1}}]^{1/2}$;
\li \>${\mathcal X}^{(m)}_{i,k-1|k-1}\!\!=S_{P_{k-1|k-1}}^{(m)}\!\xi_i+\hat x^{(m)}_{k-1|k-1}$, ($i=\overline{1,2n}$);
\li \>${\mathcal X}_{i,k-1|k-1}^{(m+1)}={\mathcal X}_{i,k-1|k-1}^{(m)} + \delta f\bigl(t_{k-1}^{(m)},{\mathcal X}_{i,k-1|k-1}^{(m)}\bigr)$; \label{formuls:start}
\li \>${\mathcal X}^{(m+1)}_{k-1|k-1}=\bigl[{\mathcal X}^{(m+1)}_{1,k-1|k-1},\ldots,{\mathcal X}^{(m+1)}_{2n,k-1|k-1}\bigr]$;
\li \>Compute $\hat x_{k-1|k-1}^{(m+1)}=\frac{1}{2n} {\mathcal X}^{(m+1)}_{k-1|k-1}{\mathbf 1}_{2n}$;
\li \>${\mathbb X}^{(m+1)}_{k-1|k-1}=\frac{1}{\sqrt{2n}}\bigl({\mathcal X}^{(m+1)}_{k-1|k-1}\!\!-{\mathbf
1}_{2n}^{\top}\!\otimes \hat x_{k-1|k-1}^{(m+1)}\bigr)$;
\li \>$P_{k-1|k-1}^{(m+1)} = {\mathbb X}^{(m+1)}_{k-1|k-1} \left[{\mathbb X}^{(m+1)}_{k-1|k-1}\right]^{\top}\!+ \delta \tilde GQ\tilde G^{\top}$; \label{con:em:p1:predict}
\li Set $\hat x_{k|k-1}: =\hat x_{k-1|k-1}^{(M)}$ and $P_{k|k-1}:=P_{k-1|k-1}^{(M)}$.
\zi \textsc{Measurement Update}: Repeat lines~\ref{con:it:MU:start}-\ref{con:it:f:p}
\zi \>of Algorithm~1a (the IT-1.5 CKF).
\end{codebox}


\end{document}